\documentclass[11pt,twoside]{article}

\usepackage{mathrsfs,amsfonts,amsmath,amssymb}
\usepackage{latexsym}
\usepackage{comment} 

\usepackage[T1]{fontenc}
\usepackage[utf8]{inputenc}
\usepackage{authblk}

\newtheorem{satz}{Theorem}
\newtheorem{proposition}[satz]{Proposition}
\newtheorem{theorem}[satz]{Theorem}
\newtheorem{lemma}[satz]{Lemma}

\newtheorem{remark}[satz]{Remark}
\newtheorem{problem}[satz]{Problem}

\newtheorem{conjecture}[satz]{Conjecture}

\def\Z{\mathbb {Z}}

\def\E{\mathsf{E}}

\def\a{\alpha}

\def\({\big (}
\def\){\big )}

\def\le{\leqslant}
\def\ge{\geqslant}
\def\_phi{\varphi}
\def\eps{\varepsilon}

\def\Gr{{\mathbf G}}

\def\gcd{\mathsf{gcd}}

\title{The Uniformity Conjecture in Additive Combinatorics}
\author[1]{I.~D.~Shkredov}
\author[2]{J.~ Solymosi}

\affil[1]{Steklov Mathematical Institute, Moscow}
\affil[2]{University of British Columbia, Vancouver}

\begin{document}
\maketitle
\begin{abstract}
In this paper we show examples for applications of the Bombieri-Lang conjecture in additive combinatorics, giving bounds on the cardinality of sumsets of squares and higher powers of integers. Using similar methods we give bounds on the sum-product problem for matchings.
\end{abstract}

\section{Introduction}
In this paper we show examples for applications of the Bombieri-Lang conjecture in additive combinatorics. This major conjecture in Diophantine geometry, if true, has far reaching consequences in number theory and other fields of mathematics \cite{HS}. We are going to use a corollary of the conjecture, the Uniformity Conjecture \cite{CHM} for special curves. There are nice applications of the Bombieri-Lang conjecture in combinatorics, for example it would imply the Erd\H{o}s-Ulam conjecture, that there is no everywhere dense subset of the real plane where the distance of any two points is a rational number \cite{Er,SZ,Ta,Sha, ABT}. 

It was also used to understand additive and multiplicative structure of finite sets of integers by Cilleruelo and Granville in \cite{CG} and by Alon, Angel, Benjamini, and Lubetzky in \cite{AA}. In this paper we continue their work improving and extending some of their results.

In Section \ref{Squares2} we are going to bound the size of the sumset of a set of squares. This is a special case of Rudin's conjecture, first stated in his seminal paper {\em "Trigonometric series with gaps"} \cite{Ru}.

Improving earlier results of Bombieri, Granville and Pintz \cite{BGP}, Bombieri and Zannier proved  that the intersection of a set of squares with any arithmetic progression $P$ does not exceed $O(|P|^{3/5+\eps})$ for any $\eps>0$ (see in \cite{BZ}). 
We obtain a conditional result on intersection of squares with generalized arithmetic progressions (all required definitions can be found in Section \ref{sec:definitions}) 
which we believe is unreachable by methods from \cite{BGP}, \cite{BZ}. 
Theorem \ref{t:md_A_intr} is a very particular case of Theorem \ref{t:md_A} below. 

\begin{theorem}
	Let $A$ be a set of squares and  $H=P_1+\dots+P_d$ be a proper generalized arithmetic progression of dimension $d$ such that $|P_j| \sim |H|^{1/d}$. 
	Then 
\begin{equation}\label{f:md_A_intr}
	|A \cap H| 
	    \ll 8^{d} |H|^{\frac{3}{4}- \frac{3}{32d+4}} \,. 
\end{equation}
	In particular, for any arithmetic progression $P$ one has 
\begin{equation}\label{f:AP_A_intr}
	|A\cap P| \ll |P|^{\frac{2}{3}} \,.
\end{equation}
\label{t:md_A_intr}
\end{theorem}

In Section \ref{Cubes} we bound sumsets of higher powers. In particular, we prove 
\begin{theorem}
    Let $A$ be a set of $k$th powers, $k>3$.
    Then any arithmetic progression of length $N$ contains at most $N^{1/2} \cdot \exp(- O(\log^{1/3} N))$ elements of $A$. 
\end{theorem}

This is better than what expected to be achievable by the methods in \cite{BGP}. 
In the following sections we are applying the Uniformity Conjecture to sum-product type problems which were introduced by Erd\H{o}s and Szemer\'edi in \cite{ESz}. Improving earlier bounds we show the following:

\begin{theorem}
Given a set of distinct pairs of integers, $M=\{(a_i,b_i) | 1\leq i\leq n\}.$ If $S=\{a_i+b_i\}$
and $P=\{a_i\cdot b_i \}$ for  $ 1\leq i\leq n,$ then  $|P|+|S|=\Omega(n^{3/5}).$    
\end{theorem}
Note that there are constructions showing examples when $|P|+|S|=O(n^{4/5+\varepsilon}).$

\section{Definitions} 
\label{sec:definitions} 

    Let $\Gr$ be an abelian group and let $A$ be a subset of $\Gr$.
Denote by $|A|$ cardinality of $A$.
In this paper we use the same letter to denote a set $A\subseteq \Gr$
and its characteristic function $A:\Gr \rightarrow \{0,1\}.$
Given two sets $A,B\subset \Gr$, define  
the {\it sumset} 
of $A$ and $B$ as 
$$A+B:=\{a+b ~:~ a\in{A},\,b\in{B}\}\,.$$
In a similar way we define the {\it difference sets} and {\it higher sumsets}, e.g., $2A-A$ is $A+A-A$. 
For an abelian group $\Gr$
the Pl\"unnecke--Ruzsa inequality (see, e.g., \cite{TV}) takes place
\begin{equation}\label{f:Pl-R} 
	|nA-mA| \le \left( \frac{|A+A|}{|A|} \right)^{n+m} \cdot |A| \,,
\end{equation} 
where $n,m$ are any positive integers. 
We  use representation function notations like  $r_{A+B} (x)$ or $r_{A-B} (x)$ and so on, which counts the number of ways $x \in \Gr$ can be expressed as a sum $a+b$ or  $a-b$ with $a\in A$, $b\in B$, respectively. 
For example, $|A| = r_{A-A} (0)$.

For $k\ge 2$ define the higher moments of convolutions  \cite{SS1} 
\[
	\E_k (A) = \sum_x r^k_{A-A} (x) = \sum_{\a_1, \dots, \a_{k-1}} |A\cap (A+\a_1) \cap \dots  \cap (A+\a_{k-1})|^2 \,,
\]
and, more generally,
\[
	\E_{k,l} (A) = \sum_{\a_1, \dots, \a_{k-1}} |A\cap (A+\a_1) \cap \dots  \cap (A+\a_{k-1})|^l 
	=
\]
\[
	= 
	\sum_{\a_1, \dots, \a_{l-1}} |A\cap (A+\a_1) \cap \dots  \cap (A+\a_{l-1})|^k = \E_{l,k} (A) \,. 
\]
    For $k=2$ we write $\E(A) := \E_2 (A)$. 
    The {\it common energy} of two sets $A,B \subseteq \Gr$ is 
$$
\E (A,B) = |\{ (a_1,a_2,b_1,b_2) \in A^2 \times B^2 ~:~ a_1 + b_1 = a_2 + b_2 \}|
\,.
$$
Thus $\E(A) = \E(A,A) = \sum_x r^2_{A+A} (x) = \sum_x r^2_{A-A} (x)$.

	Given  a set $Z\subset \Z$ denote by $Z_l$ the set $Z_l = \{ z\in Z ~:~ z \equiv 0 \pmod l\}$. 
	If $P_1,\dots, P_d \subset \Z$ be arithmetic progressions, then denote by $Q:=P_1+\dots+P_d$ the {\it generalized arithmetic progression} (GAP) of dimension $d$. A GAP $Q$ is called to be {\it proper} if $|Q| = \prod_{j=1}^d |P_j|$. For properties of generalized arithmetic progressions consult, e.g., \cite{TV}.

All logarithms are to base $2.$ The symbols  $\ll$ and $\gg$ are the usual Vinogradov's symbols, thus $a\ll b$ means $a=O(b)$ 
and $a\gg b$ is $b=O(a)$. 
If the bounds depend on some parameter $M$ {\it polynomially}, then we write $\ll_M$, $\gg_M$.
By $[n]$ denote $\{1,2,\dots, n\}$. 

\section{Preliminaries} 
\label{sec:preliminaries} 


Throughout the paper we assume a corollary of the Bombieri--Lang Conjecture by Caporaso, Harris, and Mazur \cite{CHM} -- called the Uniformity Conjecture -- in its particular form to hyper- and superelliptic curves.

\begin{conjecture}\label{uniform}
For any polynomial $f\in \Z[x]$ and $k\in \mathbb{N},$ $k\geq 2,$ if the equation 
\begin{equation}\label{conj:BL} 
	y^k = f(x) \,
\end{equation}
defines a curve with genus $g\geq 2$, then it has at most $BL_g$ rational solutions where $BL_g$ depends on genus $g$ only.
\end{conjecture}

We are going to use the K\H{o}v\'ari--S\'os--Tur\'an theorem  from graph theory \cite{KTS} several times. Often without stating it, following the standard proof of the theorem when the calculation is more sensitive to the degree distribution of the graph. It gives an upper bound on the number of edges in a bipartite graph not containing a complete bipartite subgraph $K_{s,t}.$ In our applications $s$ and $t$ are constant and the number of vertices grows. 

\begin{theorem}[K\H{o}v\'ari--S\'os--Tur\'an Theorem]\label{KST_t}
Given a bipartite graph $G(A,B)$ with two disjoint vertex sets $A$ and $B,$ so that it contains no $K_{s,t},$ $s$ vertices in $A$ and $t$ vertices in $B,$ all connected by an edge. Then the number of edges in $G(A,B)$ is $O(|B||A|^{1-1/t}+|A|).$
\end{theorem}

The first result on properties of squares which follows from \eqref{conj:BL} is, basically, \cite[Theorem 2]{CG} 
and we give the proof in our terms for the sake of completeness.

\begin{lemma}
	Let $A$ be a set of squares and let $\a,\beta,\gamma$ be different non--zero integers.
	Then 
\begin{equation}\label{f:C4}
	|A\cap (A+\a) \cap (A+\beta) \cap (A+\gamma)| \le BL_2 \,.
\end{equation}
	In particular, $\E_4 (A) \le (BL_2 + 6) |A|^4$, 
	$\E(A,S) \ll |A|^{2/3} |S|^2 + |A| |S|$, and $|A\pm A|\gg |A|^{4/3}$. 
\label{l:C4_C5}
\end{lemma}
\begin{proof}
	Indeed, for any $a\in A\cap (A+\a) \cap (A+\beta) \cap (A+\gamma)$ we find $a_1,a_2, a_3 \in A$ such that 
	$\a = a_1-a$, $\beta = a_2-a$, $\gamma = a_3-a$ and hence for $a=x^2$ we see that 
	$y^2 = (x^2+\a)(x^2+\beta)(x^2+\gamma)$ has at most $BL_2$ solutions by Conjecture \ref{uniform}. 
	It gives us \eqref{f:C4} and to obtain the required bound for $\E_4 (A)$ just notice that 
\[
	\E_4 (A) = \sum_x r^4_{A-A} (x) = \sum_{\a,\beta,\gamma} |A\cap (A+\a) \cap (A+\beta) \cap (A+\gamma)|^2 
	\le
\]
\[
	\le
		BL_2 |A|^4 + 6 \sum_{\a,\beta} |A\cap (A+\a) \cap (A+\beta) |^2 
		= BL_2 |A|^4 + 6 \E_3(A) \le (BL_2 + 6) |A|^4 \,.
\]
	In a similar way, to obtain the bound $\E(A,S) \ll |A|^{2/3} |S|^2 + |A| |S|$ we use the H\"older inequality 
\[
	\E(A,S)^3 = \left( \sum_x r^2_{A+S} (x) \right)^3 \le  (|A||S|)^2 \sum_x r^4_{A+S} (x) 
	\le
\]
\[
	\le (|A||S|)^2 \left( \sum_{x ~:~ r_{A+S} (x) \ge 5} r^4_{A+S} (x) + 64 |A| |S| \right) \,. 
\]
	It remains to estimate $\sigma := \sum_{x ~:~ r_{A+S} (x) \ge 5} r^4_{A+S} (x)$.
	We have 
\[
	\sigma 
		\le 
			\sum_{s_1,\dots,s_4} S(s_1) \dots S(s_4) \sum_{x ~:~ r_{A+S} (x) \ge 5}  A(x-s_1) \dots  A(x-s_4)
		\le
			BL_2 |S|^4 + 
\]
\[
			+4\sum_{x ~:~ r_{A+S} (x) \ge 5} r^3_{A+S} (x) 
	\le 
	BL_2 |S|^4 + 4\sigma/5 
\]
	and hence $\sigma \ll |S|^4$. 
$\hfill\Box$
\end{proof}

\begin{remark}
	In a similar way one can estimate the sum $$\sum_x A(l_1(x)) A(l_2(x)) A(l_3(x)) A(l_4(x)) \le BL_2,$$ where $l_j$ are non--proportional linear forms and $A$ is a set of squares.
	In particular, for any non--zero shift $x \in \Z$ of $A$ we have $|(A+x)(A+x)| \gg |A|^{4/3}$.  
\end{remark}

\bigskip

Now we are ready to obtain a result on the number of incidences of points and lines in $A$.
In our regime it 
works  better than the famous Szemer\'edi--Trotter Theorem \cite{sz-t}.

Given a finite set $\mathcal{L} \subset \Z^2$ put 
$\tau(\mathcal{L}) = \max_\lambda |\{ \lambda = y/x ~:~ (x,y) \in \mathcal{L},\, x\neq 0 \}|$.

\begin{lemma}
	Let $A$ be a set of squares and  $B, C, D \subset \Z$, $\mathcal{L} 
	\subset \Z^2$ be any finite sets, $B$ does not contain zero. 
	Then the number of the solutions to the equation 
\begin{equation}\label{f:incidences}
	a = bc+d \,, \quad \quad a\in A,\,~ c\in C,\,~ (b,d) \in \mathcal{L}   
\end{equation}
	is at most 
\begin{equation}\label{f:incidences_res}
	O(\tau^{1/6} (\mathcal{L}) |C| |\mathcal{L}|^{5/6} + |\mathcal{L}| ) \,. 
\end{equation}
	If $B$ is another set of squares, then  the number of the solutions to equation \eqref{f:incidences} is 
\begin{equation}\label{f:incidences_res2}
	O(\tau^{1/4} (\mathcal{L}) |C| |\mathcal{L}|^{3/4} + |\mathcal{L}| ) \,.
\end{equation}
\label{l:incidences}
\end{lemma}
\begin{proof} 
	Let $\tau = \tau(\mathcal{L})$ and let $\sigma$ be the number of the solutions to equation \eqref{f:incidences}.
	By the H\"older inequality, we have 
\[
	\sigma^6 \le |\mathcal{L}|^5 \sum_{(b,d)\in \mathcal{L}} \left(\sum_{c\in C} A(bc+d) \right)^6 
	=
\]
\begin{equation}\label{tmp:10.05_1} 
	=
	 |\mathcal{L}|^5  \sum_{c_1,\dots,c_6\in C}\, \sum_{(b,d)\in \mathcal{L}} A(bc_1+d) \dots  A(bc_6+d)  = \sigma_1 + \sigma_2 \,.
\end{equation} 
	Here the sum $\sigma_1$ corresponds to	the sum with different $c_j \in C$ 
	and 
	$\sigma_2$ is the rest.
	The arguments as in the proof of the upper bound for $\E(A,S)$ from Lemma \ref{l:C4_C5} gives us $\sigma_2 = O(|\mathcal{L}|^6)$.
	As for the sum $\sigma_1$, we see that  any tuple $bc_1+d, \dots, bc_6+d \in A$ corresponds to a solution to the equation of degree six, namely,  $\prod_{i=1}^6 (c_i+x) = y^2$ in rational numbers.
	By Conjecture \ref{uniform} the number of such solutions, i.e., the number of different $d/b$ is bounded as $BL_2$.
	But, clearly, it coins at most 
	$\tau BL_2 $
	to the sum $\sigma_1$. 
	Thus it gives us  
\[
	\sigma \ll \tau^{1/6} |C| |\mathcal{L}|^{5/6} + |\mathcal{L}| 
\]
    as required.

	To obtain \eqref{f:incidences_res2} we take the forth power instead of the sixth.
	It gives  the sum as in \eqref{tmp:10.05_1} with different $c_1,\dots,c_4$ 
\[
	\sum_{(b,d)\in \mathcal{L}} A(bc_1+d) \dots  A(bc_4+d)
	=
\]
\[
	=
	\sum_{b,d} \mathcal{L} (b,d-bc_1) A(d) A(d+b(c_2-c_1)) A(d+b(c_3-c_1)) A(d+b(c_4-c_1))
\]
	 and using the fact  that $B$ is a set of squares,  we arrive to the equation  $\prod_{i=2}^4 (x^2 + (c_i-c_1)) = y^2$ which has at most 
	 $BL_2$
	 solutions by Conjecture \ref{uniform}. 
	 Totally, we have at most $\tau BL_2$ solutions. 
$\hfill\Box$
\end{proof} 

\bigskip

\begin{remark}
	Let $A$ be a set of squares in an arithmetic progression $P = p+rj$, $j\in [k]$ and put $\lambda = \gcd(p,r)$, $p=\lambda p'$, $r=\lambda r'$.
	Now one can use formula \eqref{f:incidences_res} with $C=r\lambda^{-1}\cdot [k/s]$, $B=\lambda [s]$, 
	$D=P-r\cdot [k] = \lambda (p'+r'j)$, $-k<j<k$  and $\mathcal{L} \subset B \times D$ such that all pairs $(\lambda^{-1} b,\lambda^{-1} d)$ are coprime (here $s$ is a parameter, $s\sim k^{5/7}$).
	One can show that 
	$|\mathcal{L}| \gg |B||D|$, further $\mathcal{L}$ captures the most solutions to the equation $a=bc+d$, $a\in A$, $b\in B$, $c\in C$, $d\in D$ (see details in the proof of Theorem \ref{t:md_A} below) and thus Lemma \ref{l:incidences} 
	gives the estimate $|A| \ll k^{5/7}$ 
	which coincides with 
	the bound after Theorem 5 from \cite{CG}.  
	Formula \eqref{f:incidences_res2}  allows to obtain  a better bound, see inequality \eqref{f:AP_A} of Theorem   \ref{t:md_A} below. 
\end{remark}


One can 
see that Lemma \ref{l:incidences}, combining with the classical Burgess' method \cite{Burgess},  implies that the intersection of a set of squares with any arithmetic progression $P$ is, actually,  $O(|P|^{2/3})$
and this is slightly stronger than unconditional result from \cite{BGP} but weaker than the bound on such intersection from  \cite{BZ}.
Nevertheless, we think that our new conditional estimate  \eqref{f:md_A} is unreachable by delicate  methods from \cite{BGP}, \cite{BZ}.

\bigskip 

\begin{theorem}
	Let $A$ be a set of squares and  $H=P_1+\dots+P_d$ be a proper generalized arithmetic progression of dimension $d$.
	Suppose that there is $1\le m\le d$ arithmetic progressions $P_j$ with  $|P_j| \gg |H|^{\frac{3}{4m+2}}$ and denote by $R_m$ 
	size of the rest of $H$ (if the rest is empty, then put $R_m=1$).
	Then 
\begin{equation}\label{f:md_A}
	|A \cap H| 
	    \ll 8^{d} |H|^{\frac{3}{4}} \cdot \left( \frac{R^{2m+1}_m}{|H|^{\frac34}} \right)^{\frac{1}{8m+1}} \,. 
\end{equation}
	In particular, for any arithmetic progression $P$ one has 
\begin{equation}\label{f:AP_A}
	|A\cap P| \ll |P|^{\frac{2}{3}} \,.
\end{equation}
\label{t:md_A}
\end{theorem}
\begin{proof}
	We can assume that $A\subseteq H$ and moreover splitting $H$ onto $2^d$ parts we can assume that not only $H$ but even  $H+H$ is a proper generalized  arithmetic progression. 
	It will coast us the factor $2^d$ in our final bound \eqref{f:md_A}. 

    Let $\left\lceil |H|^{\frac{3}{4m+2}} \right\rceil \le M \le (|H|/2)^{1/d}$ be a parameter which we will choose later, and let $\eps = 1/M$.
	Also, let $H= P_1+\dots +P_d = k_0 + \sum_{j=1}^d k_j x_j$, $0\le x_j < |P_j|$.
	Take all progressions $P_j$,  having sizes greater than $M$. 
	Without loosing of the generality we can assume that we have first $m\ge 1$ of such progressions and fix elements of the progressions $P_{m+1},\dots, P_d$ from the rest. 
	In this case the shift $k_0$ can be changed but we use the same letter $k_0$ for this number.  
	Also, we use the  letter $A'$ to the subset of $A$ in this generalized  arithmetic progression of dimension $m$. 
	Put $\lambda = \gcd(k_0,k_1, \dots, k_m)$ and let $k_j = \lambda k'_j$. 
	Notice that $\gcd(k'_0,k'_1, \dots, k'_m) = 1$.
    Shrinking these progressions $P_j$, $j\in [m]$ in $\eps$ times, we obtain  a new proper generalized  arithmetic progression $Q = \sum_{j=1}^m k'_j x_j$ of size $|Q| =  \prod_{j=1}^m [\eps |P_j|]$.  
	Also, let $I\subseteq [M]$ be the set of squares in $[M]$, $|I| \gg M^{1/2}$. 
	We have $A'+ Q \cdot \lambda I \subseteq H' = k_0 + \sum_{j=1}^m k_j x_j$,
	$0\le x_j < 2|P_j|-1$.
	In other words, $H'$ is a new generalized  arithmetic progression of dimension $m$ such that  $|H'| \le 2^m \prod_{j=1}^m |P_j|$.  
	The number of the solutions to the equation $a+q\cdot \lambda i=h'$, $a\in A'$, $q\in Q$, 
	$i\in I$, $h'\in H'$ 
	is 
	exactly  
    $|A'||Q||I|$. 
	Split the set $I \times H'$ onto the sets $\mathcal{L}_l$ such that for any pair $(i,h') \in \mathcal{L}_l$ we have $\gcd(i,\lambda^{-1} h') = l$. 
    Let $H'' = \lambda^{-1} H' = k'_0 + \sum_{j=1}^m k'_j x_j$,
	$0\le x_j < 2|P_j|-1$.
	Now, clearly, we have   $|I_l| = |I|/l + O(1)$. 
	Let us calculate $H''_l$, $l\le M$.  
	To do this 
	consider the equation 
	\begin{equation}\label{f:H_l}
	    k'_0 + \sum_{j=1}^m k'_j x_j \equiv 0 \pmod l    \,.
	\end{equation} 
	If $x_j$ run over $\Z/l\Z$, then by the Chinese remainder theorem the function $q(l)$ counting the number of the solutions to equation \eqref{f:H_l} is multiplicative.    
	We know that $\gcd(k'_0,k'_1, \dots, k'_m) = 1$  and that $H+H$ is a proper generalized  arithmetic progression. 
	Hence equation \eqref{f:H_l} has 
	\[
	    l^{m-1}  \prod_{j=1}^m \left(\frac{|P_j|}{l} + O(1) \right) 
	    = \frac{|H''|}{l} + O\left(\frac{d|H''|}{M} \right) 
	\]
	solutions.
	It is easy to see that the numbers $|\mathcal{L}_l|$ decrease as $O(|\mathcal{L}_1|/l^2)$.
	Indeed, e.g.,  by the M\"obius transform, we have 
	$$
	    |\mathcal{L}_l| = \sum_t \mu (t) |I_{lt}| |H''_{lt}| = 
	    \sum_t \mu (t) \left( \frac{|I|}{lt} + O(1) \right) \left( \frac{|H''|}{lt} + O\left(\frac{d|H''|}{M} \right) \right)
	    =
	$$
	\begin{equation}\label{tmp:16.05_1}
	    =
	    \frac{1}{l^2} \sum_t \mu (t) \frac{|I| |H''|}{t^2} + O(d|H''|) 
	    =
	    \frac{|\mathcal{L}_1|}{l^2} + O(d|H''|) \,.
	\end{equation} 
	Using Lemma \ref{l:incidences} and estimate \eqref{tmp:16.05_1}, we get
\[
	|A'||Q||I| 
	\ll |Q| \sum_l |\mathcal{L}_l|^{3/4} + |I| |H'|  
	\ll
	|Q| (|I| |H'|)^{3/4} + |I| |H'| 
\]
	and hence
\[
	|A'| \ll \frac{2^{3m/4}}{M^{1/8}} \prod_{j=1}^m |P_j|^{3/4}  + 2^{2m} M^m \,. 
\]
    Returning to our initial set $A$, we see that  
\begin{equation}\label{f:|A|_final_md}
    |A| \le \left( \frac{2^{3m/4}}{M^{1/8}} \prod_{j=1}^m |P_j|^{3/4}  + 2^{2m} M^m \right) 
    R_m 
    \le
    \frac{2^{3d/4} |H|^{3/4} R^{1/4}_m}{M^{1/8}} + 2^{2d} M^{m} R_m \,.
\end{equation}
    We can assume that $R_m \le |H|^{\frac{3}{8m+4}}$ because otherwise the result is trivial in view of Lemma \ref{l:C4_C5} and the fact that $|H+H| < 2^d |H|$. 
	Now taking 
    $$
    M = \left\lceil (|H|/R_m)^{\frac{6}{1+8m}}\right\rceil \ge \left\lceil |H|^{\frac{3}{4m+2}} \right\rceil 
    $$
	we obtain 
\begin{equation}\label{f:|A|_final_md'}
	|A| \ll 4^{d} |H|^{\frac{3}{4}} \cdot \left( \frac{R^{2m+1}_m}{|H|^{\frac34}} \right)^{\frac{1}{8m+1}} 
\end{equation}
	as required.

	In the case when $H$ is an arithmetic progression we do not need to split $H$ onto $Q$ and onto the rest, hence $m=d=1$, $R_d=1$ and
	estimates \eqref{f:|A|_final_md}, \eqref{f:|A|_final_md'} give the required bound.
	This completes the proof. 
$\hfill\Box$
\end{proof}

\bigskip

Theorem  \ref{t:md_A} works for relatively small $d$ or, more precisely, for generalized arithmetic progressions with relatively large sides.
We use the arguments from the proof of Theorem  \ref{t:md_A} in the next section.


\section{New bound for the energy of squares}\label{Squares2}

We need in two auxiliary results. 
The first one is \cite[Lemma 13]{SS_Balog} about multiplicative structures contained in additively rich sets. 
This result depends on our knowledge about the Polynomial Freiman--Ruzsa Conjecture, see \cite{Sanders_2A-2A}, \cite{Sanders_3log}.

\begin{lemma}
	Let $A$ be a subset of an abelian  ring such that  $|A+A| \le K|A|$.
	Then there exists an absolute constant $C >0$
	such that for any positive integers  $d\ge 2$ and $l$ there is 
	a set $Z$	of size  $|Z| \ge   \exp \left(-C l^{3} d^2 \log^2 K \right) |A|$ with
	\begin{equation}\label{f:mult_inclusion_3}	
	[d^l] \cdot Z \subseteq 2A-2A \,.
	\end{equation} 
\label{l:mult_2A-2A}
\end{lemma}

The second result is a special case of \cite[Theorem 6.3]{s_mixed}.

\begin{theorem}
	Let $A\subseteq \Gr$ be a set, $\E(A) = |A|^3/K$, $\E_4 (A) = M|A|^5/K^3$.
	Then there is an absolute constant $C>0$ and a set $A'\subseteq A$, $|A'|\gg |A|/M^C$ such that for any positive integers $n,m$ one has 
	$|nA'-mA'| \ll M^{C(n+m)} K|A'|$. 
\label{t:struct_E2E4}
\end{theorem}

Now we are ready to prove the main result of this section.

\begin{theorem}
	Let $A$ be a set of squares.
	Then for any $\alpha < \frac{1}{11}$ one has 
\[
 	\E(A) \ll |A|^{8/3} \cdot \exp(-O(\log^{\alpha} |A|)) \,.
\]
\label{t:E_squares}
\end{theorem}
\begin{proof} 
	Write $\E(A) = |A|^3/K$, $\E_4 (A) = M|A|^5/K^3$ and we know by Lemma \ref{l:C4_C5} that $\E_4 (A) \ll |A|^4$ as well as $K\gg |A|^{1/3}$. 
	Hence $M\ll K^3 |A|^{-1}$. 
	If $M\gg \exp (C_0 \log^{\alpha} |A| )$, then we are done. 
	Here and below $C_j$ are absolute positive  constants.  
	Applying Theorem \ref{t:struct_E2E4}, we find a set $A'\subseteq A$, $|A'|\gg |A|/M^C$ such that for any positive integers $n,m$ one has 
	$|nA'-mA'| \ll M^{C(n+m)} K|A'|$, where $C>0$ is an absolute constant from Theorem \ref{t:struct_E2E4}. 
	Now using Lemma \ref{l:mult_2A-2A} with $d=2$, a parameter $l \sim  (\log |A| /\log^2 M )^{1/3}$ and $A=A'-A':=D'$ we find a set $Z$, $|Z| \ge   \exp \left(-C_1 l^{3} \log^2 M \right) |D'|$ with
$
	[2^l] \cdot Z \subseteq 2D'-2D' \,.
$
	Putting $B$ be squares of $[2^l]$, $|B| \gg 2^{l/2}$,  $C=Z$ and $D=3D'-2D'$, we obtain exactly $|B||C||A'|$  solutions to the equation $d=a'+bz$, where $a'\in A'$, $d\in D$, $b\in B$, $z\in Z$.   
	Now by the main result of  \cite{Sanders_3log} the set $D$ contains a proper GAP, say,  $H$ of size $|D|\exp(-C_* \log^{\kappa} M)$, where $\kappa > 3$ is any constant. 
	Hence applying 
	the 
	convering lemma from \cite[Exercise 1.1.8]{TV}, we find $X$, $|X| \ll \exp(C_* \log^{\kappa} M) \log |A|$ such that $D\subseteq H+X$. 
	Then we have 
\begin{equation}\label{tmp:20.05_1}
    |B||Z||A'| \le \sum_{x\in X} |\{ d=a'+bz ~:~ a'\in A', d\in H+x, b\in B, z\in Z \}| := \sum_{x\in X} \sigma(x) \,.
\end{equation}
    Fix $x\in X$ and estimate each $\sigma (x)$ separately. 
    As in the proof of Theorem \ref{t:md_A} split the set $B\times (H+x)$ onto the sets $\mathcal{L}_l$ such that for any pair $(b,h+x) \in \mathcal{L}_l$ we have $\gcd(b,\lambda^{-1} (h+x)) = l$, where $\lambda = \lambda (x)$ is the correspondent $\gcd$ of steps of $H+x$.   
	Using the arguments as in the proof of Theorem \ref{t:md_A} and applying  the second part of  Lemma \ref{l:incidences}, we have 
\[
    \sigma(x) \ll |B| |H| +  |Z| \sum_l |\mathcal{L}_l|^{3/4} \ll |B| |H| +  |Z| (|B||H|)^{3/4} \,.
\]
	Returning to \eqref{tmp:20.05_1}, we obtain 
\[
	|B||Z||A'| \ll |X| |B| |D| +  |X| |Z| (|B||D|)^{3/4} \ll |X| |Z| (|B||D|)^{3/4}  \,.
\] 
   Here we have used our choice of the parameter $l$.
	In other words, by the Pl\"unnecke--Ruzsa inequality \eqref{f:Pl-R} and the bound $M \ll \exp (C_0 \log^{\alpha} |A| )$ as well as the condition $\alpha <\frac{1}{11}$, we get 
\[
	2^{l/2} |A| \ll M^{C_2} |X|^4 K^3 \ll 2^{l/4} K^3 
\]
	Hence using the bound $M \ll \exp (C_0 \log^{\alpha} |A| )$ and our choice of $l$, we derive 
\[
	K\gg 2^{\frac{l}{12}} |A|^{\frac{1}{3}} \gg |A|^{\frac{1}{3}} \cdot \exp(O(\log^{\frac{3}{11}} |A|)) 
\]
and this is even better than required. 
$\hfill\Box$
\end{proof} 

\medskip

\section{Another additive problem for squares}
More than 50 years ago Paul Erd\H{o}s asked the following question in \cite[Problem 40]{E_63}. Are there integers $a_1,\ldots a_k$ such that $a_i+a_j$ is a square for any $1\leq i<j\leq k$? There are other variants of the question like several problems in sections D14 and D15 in Richard Guy's problem book \cite{G}.

As the curve $y^2=(x+a_1)(x+a_2)(x+a_3)(x+a_4)(x+a_5)$ has genus two, by Conjecture \ref{uniform} we have $k\leq BL_2.$ In a more general statement we can show the following.

\begin{theorem}
If $A$ and $B$ are two sets of integers so that $5\leq|A|\leq|B|=n$ then the number of $a\in A$,  $b\in B$ pairs such that $a+b$ is a square is $O(|A||B|^{4/5}+|B|).$
\end{theorem}

Indeed, let us consider the bipartite graph $G(A,B)$ where two vertices $a\in A$, $b\in B$ connect by an edge iff $a+b$ is a square. By Conjecture \ref{uniform} it contains no $K_{5,BL_2+1},$ so by Theorem  \ref{KST_t}  the number of edges -- and therefore the number of pairs adding to a square -- is $O(|A||B|^{1-1/5}+|B|).$

\section{Higher powers}\label{Cubes}

In \cite{BGP} the authors discuss the problem of determining the maximal number $Q_k(N)$ of $k$th powers in an arithmetic progression of length $N$.
Using a more deep structural result \cite[Theorem 6.1]{s_mixed} instead of Theorem \ref{t:struct_E2E4}, we obtain an analogue of Theorem  \ref{t:E_squares} for cubes. Now the curves we are working with are of the form $y^3 = (x^3+a)(x^3+b),$ which are genus 4 curves.

\begin{theorem}
  	Let $A$ be a set of cubes.
	Then  for any $\alpha < \frac{1}{11}$ one has 
\[
 	\E(A) \ll |A|^{5/2} \cdot \exp(-O(\log^{\alpha} |A|)) \,.
\]
    In particular, $Q_3 (N) \ll N^{2/3} \cdot \exp(-O(\log^{\alpha} |A|))$.  
\label{t:E_cubes}  
\end{theorem}
	Indeed, we show (sketch) that  if $\E(A) = |A|^3/K$, $\E_3 (A) = M|A|^4/K^2$, 
	then $M\gg \exp (C_0 \log^{\alpha} |A|)$.
	Here and below $C_j$ are absolute positive  constants. 
	We know by an analogue of Lemma \ref{l:C4_C5} for cubes that $\E_3 (A) \ll |A|^3$ as well as $K\gg |A|^{1/2}$. 
	Hence $M\ll K^2 |A|^{-1}$. 
	If $M\gg \exp (C_0 \log^{\alpha} |A|)$, then we are done. 
	Applying \cite[Theorem 6.1]{s_mixed}, we find a set $A'\subseteq A$, $|A'|\gg |A|/M^{C_1}$ such that for any positive integers $n,m$ one has 
	$|nA'-mA'| \ll M^{C_1 (n+m)} K|A'|$.
	After that repeat the arguments of the proof of Theorem \ref{t:E_squares}.

\bigskip

Clearly, we have an  analogue of Theorem  \ref{t:md_A} about intersections of cubes with generalized  arithmetic progressions.

Now let $k>3$. In this case by Conjecture \ref{uniform} the equation $y^k = x^k + a$ has a uniformly bounded number of the solutions. (For $k\geq 3$ the genus of the $y^k = x^k + a$ curve is $(k-1)(k-2)/2$ since it smooth if $a\neq 0$.) There are other related conjectures implying the uniform bound for the $k>3$ case. For example, the equation $x^5+y^5=u^5+v^5$ has no known nontrivial solution and it is expected that $x^5+y^5=n$ has $O(1)$ integer solutions for any $n\in \mathbb{N}$ . (Note that $x^3+y^3=n$ has an unbounded number of solutions \cite{CCC}.)

From there, $\E(A) \ll |A|^2$ and this is optimal. 
Nevertheless, it is possible to obtain a new upper bound for $Q_k (N)$ which breaks the square--root barrier.

\begin{theorem}
    Let $A$ be a set of $k$th powers, $k>3$ 
    and $H$ be a set, $|H+H|\le K|A|$. 
    Let $\alpha < \frac{1}{11}$ be any number and let 
\begin{equation}\label{cond:Qk}
    \log K \ll \log^{\alpha} |H| \,.
\end{equation}
    Then
\begin{equation}\label{f:Qk}
    |A \cap H| \ll \sqrt{|H|} \cdot \exp(-O(\log |H| /\log^2 K )^{1/3}) \,.
\end{equation}
    In particular, 
    $Q_k (N) \ll N^{1/2} \cdot \exp(- O(\log^{1/3} N))$. 
\label{t:Qk}
\end{theorem}

\begin{proof}
  	Without loosing of the generality we can assume that $A\subseteq H$. 
  	Let $X = \exp(\log^{\kappa} K) \log^2 |A|$, where $\kappa > 3$ is any number. 
  	Using Lemma \ref{l:mult_2A-2A} with $d=2$, a parameter $l \sim  (\log |H| /\log^2 K )^{1/3}$ and 
  	$A=H$ we find a set $Z$, $|Z| \ge   \exp \left(-C_1 l^{3} \log^2 K \right) |H|$ with
$
	[2^l] \cdot Z \subseteq 2H-2H \,.
$
	Putting $B$ be $k$th powers of  $[2^l]$, $|B| \gg 2^{l/k}$, 
	$C=Z$ and $D=3H-2H$, we obtain exactly $\sigma:= |B||Z||A|$  solutions to the equation $d=a+bz$, where $a\in A$, $d\in D$, $b\in B$, $z\in Z$.
	On the other hand, applying the arguments  as in the proof of Lemma \ref{l:incidences} of Theorem \ref{t:E_squares}, we obtain in view of Sanders' Theorem \cite{Sanders_3log} that 
\[
    X^{-2} \sigma^2 \ll (|B||D|)^2 + |B||D| \sum_{z_1 \neq z_2 \in Z} \sum_{b,d} A(d) A(d+b(z_2-z_1)) 
    \ll 
    (|B||D|)^2 + |B||D||Z|^2 
\]
    because we arrive to the equation $y^k = x^k + z_2-z_1$ which has at most $BL_\ell$ solutions by Conjecture \ref{uniform}, where $\ell=(k-1)(k-2)/2. $
    In view of our choice of the parameter $l$ it gives us
\[
    X^{-2} (|B||Z||A|)^2 \ll (|B||D|)^2 + |B||D||Z|^2 
    \ll |B||D||Z|^2 \,.
\]
    By the Pl\"unnecke--Ruzsa inequality $|D|\le K^5 |H|$ and hence thanks to our condition \eqref{cond:Qk}, we derive 
\[
    |A| \ll \sqrt{|H|}\cdot X \exp(-O(\log |H| /\log^2 K )^{1/3})
    \ll
    \sqrt{|H|}\cdot \exp(-O(\log |H| /\log^2 K )^{1/3})
\]
as required. 
$\hfill\Box$
\end{proof} 

\section{Sum-Product along edges}

The sum-product problem was introduced by Erd\H{o}s and Szemer\'edi in \cite{ESz}. A particular variant of the problem which was raised there is the following:
\begin{problem}
Given two $n$-element sets of integers, $A=\{a_1,\ldots ,a_n\}$ and $B=\{b_1,\ldots ,b_n\} $ let us define sumset and product set as 
$$S=\{a_i+b_i |  1\leq i\leq n\}~~\text{     }
and~~\text{     }
 P=\{a_i\cdot b_i |  1\leq i\leq n\}.$$
Erd\H{o}s and Szemer\'edi conjectured that 
\begin{equation}\label{SZP}
|P|+|S|=\Omega(n^{1/2+c})
\end{equation}
for some constant $c>0.$ 
\end{problem} 
The problem was connected to the sum of squares problem of Rudin, as we discussed above, by Chang in \cite{MCC}.
Alon et al.  proved in  \cite{AA} that under the assumption of Conjecture \ref{uniform}, one can take $c=1/14$ in equation (\ref{SZP}), i.e. $|P|+|S|=\Omega(n^{4/7}).$

\medskip

We will show that one can take $c=1/10$ in equation (\ref{SZP}), even if we allow $A$ and $B$ to be multi-sets.   

\medskip

\begin{theorem}\label{matching}
Given a set of distinct pairs of integers, $M=\{(a_i,b_i) | 1\leq i\leq n\}.$ If $P$ and $S$ are defined as above, then $|P|+|S|=\Omega(n^{3/5}).$    
\end{theorem}

\begin{proof}Let us define a bipartite graph $G(P,S)$ where two vertices $p\in P$ and $s\in S$ are connected by an edge if there is an $i,  1\leq i\leq n,$ so that $p=a_i\cdot b_i$ and $s=a_i+b_i.$ The number of edges in $G(P,S)$ is at least $n/2,$ since $a_i$ and $b_i$ are determined by the two equations as solutions of a quadratic equation. 

In $G(P,S)$ there is no $K_{C,3},$ a complete bipartite graph between vertices $p_1,p_2,p_3\in P$ and $s_1,\ldots s_C\in S$ since the identity $(a_i+b_i)^2-4a_i\cdot b_i=(a_i-b_i)^2$ implies that these numbers would give $C$ solutions to the $y^2=(x^2-4p_1)(x^2-4p_2)(x^2-4p_3)$ equation, contradicting Conjecture \ref{uniform} if $C$ is large enough.

\medskip
It gives a bound on the cubic energy

$$
\sum_{s\in S} \mathrm{deg}_{G(P,S)}^3 (s)= O\left({|P|}^{3}\right) \,.
$$

So by the H\"older inequality and the inequality $\sum_{s\in S} \mathrm{deg}_{G(P,S)} \ge n/2$,  we have $n^3=O(|P|^3|S|^2).$ 
$\hfill\Box$
\end{proof}

\medskip

The construction in \cite[Theorem 3]{ARS} shows that there are examples for sets $M$ where $|P|+|S|=O(n^{4/5+\varepsilon}),$ so Theorem \ref{matching} can not be improved beyond an extra $1/5$ in the exponent. (The construction in \cite[Theorem 3]{ARS} describes a graph and sums and products along the edges of the graph, but simply separating the $e$ edges into a matching with $e$ pairs we get the desired example for $M.$)

\medskip

\section{Appendix}

In this section we obtain  a 
purely 
combinatorial result about the difference set for sets $A$ which have the sets of the form 	$A\cap (A+\a_1) \cap \dots \cap (A+\a_{k-1})$
to be uniformly bounded.
It shows that such sets either grow faster than the ordinary application of the H\"older gives us or they must have a rich additive structure.

\begin{proposition}
    Let $k\ge 3$ be a positive integer, $C>0$ be an absolute constant,  and let $\Gr$ be an abelian group such that any non--zero element of $\Gr$ has order at least $k$. 
	Also,  let 
	$A \subseteq \Gr$ be a set such that for any different non--zero $\a_1, \dots, \a_{k-1} \in \Gr$ one has 
	\begin{equation}\label{cond:C_comb}
		|A\cap (A+\a_1) \cap \dots \cap (A+\a_{k-1})| \le C \,.
	\end{equation}
	Then for $P \subseteq D:=A-A$, $P = \{s ~:~ r_{A-A} (s) \ge |A|^2/(4|D|) \}$ one has $|D| \gg |A|^{\frac{k}{k-1}}$ and 
	\[
	\sum_{x\in D} r_{P-P} (x) \gg_k |A|^{\frac{2k}{k-1}} \,.
	\]
	\label{p:C_comb}
\end{proposition}
\begin{proof}
	We assume that $k\ge 4$ because for $k=3$ the result is known, see \cite{SS1}. 
	We have 
	\[
	\sum_{s,t} |A\cap (A+s) \cap (A+t)|^k = \E_{3,k} (A) = \E_{k,3} (A) = \sum_{\a_1, \dots, \a_{k-1}} |A\cap (A+\a_1) \cap \dots  \cap (A+\a_{k-1})|^3 
	\le
	\]
	\begin{equation}\label{tmp:04.05_1}
	\le
	C^2 |A|^k 
	+
	\binom{k}{2} \binom{k-1}{2} |A|^k	 
	+
	\binom{k}{2} \sum^*_{\a_1, \dots, \a_{k-2}} |A\cap (A+\a_1) \cap \dots  \cap (A+\a_{k-2})|^3 \,,
	\end{equation}
	where the sum above is taken over different non--zero shifts $\a_1, \dots, \a_{k-2}$. 
	Expanding this sum, we obtain
	\begin{equation}\label{tmp:21.05_1}
	\sum^*_{\a_1, \dots, \a_{k-2}}\, \sum_{y,z} \sum_x A(x) A(x-\a_1) \dots A(x-\a_{k-2}) A(x+y) A(x+y-\a_1) \dots A(x+y-\a_{k-2}) \times
	\end{equation} 
	\begin{equation}\label{tmp:21.05_2}
	\times 
	A(x+z) A(x+z-\a_1) \dots A(x+z-\a_{k-2}) \,.
	\end{equation} 
	By the assumption any non--zero element of $\Gr$   has order at least $k$. 
	Hence if $y$ or $z$ are non--zero, then the sum in \eqref{tmp:21.05_1}, \eqref{tmp:21.05_2} contains at least $k-1$ different shifts and thus the sum over $x$ is at most $C$ thanks to \eqref{cond:C_comb}.
	If not, then the sum 
	can be estimated as  
	$|A|^{k-1}$. 
	Thus 
	\[
	\E_{3,k} (A) \le C^2 |A|^k + \binom{k}{2} \binom{k-1}{2} |A|^k + \binom{k}{2} (2C |A|^k  + |A|^{k-1}) \ll_k |A|^k \,.	 
	\]
	Now if $A\cap (A+s) \cap (A+t) \neq \emptyset$, then $s,t, s-t \in D$.
	Hence by the H\"older inequality 
	\[
	|A|^{3k} \ll \left( \sum_{s,t\in P} |A\cap (A+s) \cap (A+t)| \right)^k \le \E_{3,k} (A) \left(\sum_{x\in D} r_{P-P} (x) \right)^{k-1}
	\]
	\[
	\ll_k
	|A|^k \left(\sum_{x\in D} r_{P-P} (x) \right)^{k-1}
	\]
	as required. 
	$\hfill\Box$
\end{proof}

\bigskip

{\bf Question.} Let $A$ be a set of squares. Is it possible to prove that the equation $x-y=z$, $x,y,z\in A-A$ has 
$|A-A|^{2-c}$ solutions?  
Here $c>0$ is an absolute constant.

\section{Acknowledgements} 
The research of the second author was supported in part by an NSERC Discovery grant and OTKA NK grant. The work of the second author was also supported by the European Research Council (ERC) under the European Union's Horizon 2020 research and innovation programme   (grant agreement No. 741420, 617747, 648017).

\end{document}